\documentclass[a4paper,10pt,oneside,headsepline]{scrartcl}

\usepackage{etoolbox}
\usepackage{ifdraft}
\usepackage{iftex}

\ifluatex
\usepackage[utf8]{luainputenc}
\usepackage[T1]{fontenc}
\fi

\usepackage{lmodern}
\usepackage{fourier}

\ifluatex
\usepackage{fontspec}
\fi
\ifxetex
\usepackage{fontspec}
\fi
\usepackage{newunicodechar}

\usepackage{amssymb, amsfonts}

\usepackage{mathrsfs} 
\usepackage{dsfont}
\usepackage[usenames,dvipsnames]{xcolor}
\usepackage{lettrine}
\usepackage{url}

\usepackage{stmaryrd}

\usepackage[english]{babel}

\usepackage[draft=false]{scrlayer-scrpage}

\usepackage{wrapfig}
\usepackage{footmisc}
\usepackage{needspace}

\usepackage{amsmath}
\usepackage{nccmath}
\usepackage[thmmarks, amsmath, amsthm]{ntheorem}

\providebool{nobiblatex}
\boolfalse{nobiblatex}
\ifbool{nobiblatex}{%
}{%
  \usepackage[backend=biber,style=numeric-comp,giveninits,url=false,sortcites=true,maxbibnames=99]{biblatex}
  \addbibresource{bibliography.bib}
}

\usepackage[final,
            pdfborder={0 0 0}, colorlinks=true,
            linkcolor=BrickRed, citecolor=ForestGreen, urlcolor=RoyalBlue]{hyperref}

\ifdraft{%
\usepackage[textsize=scriptsize,colorinlistoftodos]{todonotes}
\usepackage{lineno}
}{}

\usepackage[ruled]{algorithm2e}
\usepackage{booktabs}
\usepackage[inline,shortlabels]{enumitem}
\usepackage{float}
\usepackage{graphicx}
\usepackage{multicol}
\usepackage{scrtime}
\usepackage{tikz}
\usetikzlibrary{matrix,arrows,positioning}
\usepackage{tikz-cd}
\usepackage{adjustbox}

\ifdraft{\date{\today}}{\date{}}

\ifdraft{%
\linenumbers
\newcommand*\patchAmsMathEnvironmentForLineno[1]{%
  \expandafter\let\csname old#1\expandafter\endcsname\csname #1\endcsname
  \expandafter\let\csname oldend#1\expandafter\endcsname\csname end#1\endcsname
  \renewenvironment{#1}%
     {\linenomath\csname old#1\endcsname}%
     {\csname oldend#1\endcsname\endlinenomath}}%
\newcommand*\patchBothAmsMathEnvironmentsForLineno[1]{%
  \patchAmsMathEnvironmentForLineno{#1}%
  \patchAmsMathEnvironmentForLineno{#1*}}%
\AtBeginDocument{%
\patchBothAmsMathEnvironmentsForLineno{equation}%
\patchBothAmsMathEnvironmentsForLineno{align}%
\patchBothAmsMathEnvironmentsForLineno{flalign}%
\patchBothAmsMathEnvironmentsForLineno{alignat}%
\patchBothAmsMathEnvironmentsForLineno{gather}%
\patchBothAmsMathEnvironmentsForLineno{multline}%
}}

\KOMAoptions{DIV=10}

\clearmainofpairofpagestyles
\automark[section]{subsection}

\renewcommand{\subsectionmark}[1]{}

\lohead{{\small \headertitle}}
\rohead{{\small \headerauthors}}

\RedeclareSectionCommand[%
  font=\Large\sffamily\bfseries,%
  beforeskip=1\baselineskip,%
  afterskip=0.5\baselineskip,%
  indent=0em%
  ]{section}

\RedeclareSectionCommands[%
  font=\normalfont\bfseries,%
  afterskip=-1em%
  ]{subsection,subsubsection}

\RedeclareSectionCommands[%
  font=\normalfont\itshape,%
  afterskip=-1em,%
  indent=0pt,
  ]{paragraph}

\ifbool{nobiblatex}{}{%
  \AtEveryBibitem{\clearfield{doi}}
  \AtEveryBibitem{\clearfield{isbn}}
  \AtEveryBibitem{\clearfield{issn}}
  \AtEveryBibitem{\clearfield{pages}}
  \AtEveryBibitem{\clearlist{language}}

  \setlength\bibitemsep{1pt}
  
  \renewbibmacro*{in:}{}
  \DeclareFieldFormat
    [article,inbook,incollection,inproceedings,patent,thesis,unpublished]
    {title}{#1}
}

\newenvironment{enumeratearabic*}{
\begin{enumerate*}[label=(\arabic*)] 
}{
\end{enumerate*}
}

\newenvironment{enumerateroman*}{
\begin{enumerate*}[label=(\roman*)] 
}{
\end{enumerate*}
}

\numberwithin{equation}{section}

\theoremnumbering{arabic}
\newtheorem{theoremcounter}{theoremcounter}[section]
\theoremnumbering{Alph}
\newtheorem{maintheoremcounter}{maintheoremcounter}

\theoremstyle{plain}

\newtheorem{proposition}[theoremcounter]{Proposition}
\newtheorem{theorem}[theoremcounter]{Theorem}
\newtheorem*{untheorem}{Theorem}
\theoremstyle{plain}

\newtheorem{maintheorem}[maintheoremcounter]{Theorem}

\theoremstyle{definition}

\theoremstyle{remark}

\newtheorem{remark}[theoremcounter]{Remark}

\theoremstyle{nonumberremark}

\newtheorem{mainremark}{Remark}

\newcommand{\tx}{\ensuremath{\text}}

\ifdraft{
 \newcommand{\texpdf}[2]{#1}
}{
 \newcommand{\texpdf}[2]{\texorpdfstring{#1}{#2}}
}

\newcommand{\tbf}{\bfseries}

\newcommand{\cE}{\ensuremath{\mathcal{E}}}

\newcommand{\rmB}{\ensuremath{\mathrm{B}}}

\newcommand{\rmE}{\ensuremath{\mathrm{E}}}

\newcommand{\rmH}{\ensuremath{\mathrm{H}}}

\newcommand{\rmL}{\ensuremath{\mathrm{L}}}
\newcommand{\rmM}{\ensuremath{\mathrm{M}}}

\newcommand{\rmS}{\ensuremath{\mathrm{S}}}

\newcommand{\rmZ}{\ensuremath{\mathrm{Z}}}

\newcommand{\wtd}{\widetilde}
\newcommand{\ov}{\overline}

\newcommand{\lra}{\ensuremath{\longrightarrow}}

\newcommand{\lmto}{\ensuremath{\longmapsto}}

\renewcommand{\Re}{\ensuremath{\mathrm{Re}}}
\renewcommand{\Im}{\ensuremath{\mathrm{Im}}}

\newcommand{\Hom}{\ensuremath{\mathrm{Hom}}}

\newenvironment{psmatrix}{\left(\begin{smallmatrix}}{\end{smallmatrix}\right)}

\newcommand{\ZZ}{\ensuremath{\mathbb{Z}}}
\newcommand{\QQ}{\ensuremath{\mathbb{Q}}}
\newcommand{\RR}{\ensuremath{\mathbb{R}}}
\newcommand{\CC}{\ensuremath{\mathbb{C}}}

\newcommand{\GL}[1]{\ensuremath{\mathrm{GL}_{#1}}}

\newcommand{\SL}[1]{\ensuremath{\mathrm{SL}_{#1}}}

\newcommand{\sym}{\ensuremath{\mathrm{sym}}}

\newcommand{\HS}{\mathbb{H}}

\mathchardef\xboxplus=\numexpr \boxplus-"2000\relax
\renewcommand{\boxplus}{\mathbin{\mathop{\xboxplus}}}

\newcommand{\sfd}{\mathsf{d}}
\renewcommand{\sym}[1]{\mathrm{sym}^{#1}}
\newcommand{\symd}{\sym{\sfd}}

\newcommand{\bfone}{\mathbf{1}}

\renewcommand{\SL}[2]{\mathrm{SL}_{#1}({#2})}
\renewcommand{\GL}[1]{\mathrm{GL}({#1})}
\newcommand{\GLbig}[1]{\mathrm{GL}\big({#1}\big)}
\newcommand{\Ga}{\Gamma}
\newcommand{\ga}{\gamma}
\renewcommand{\a}{\alpha}

\renewcommand{\d}[1]{\mathop{}\mathrm{d}{#1}}
\newcommand{\bs}{\backslash}
\newcommand{\Ext}{\mathrm{Ext}}

\renewcommand{\Im}{\mathrm{Im}}
\renewcommand{\Re}{\mathrm{Re}}
\newcommand{\bp}{\boxplus}

\newcommand{\ol}{\overline}

\newcommand{\llangle}{\langle\!\langle}
\newcommand{\rrangle}{\rangle\!\rangle}

\newcommand{\headertitle}{{\normalfont%
  Depth two Eichler-Shimura Integrals of Cusp Forms
}}
\newcommand{\headerauthors}{%
  T.~Magnusson,
  M.~Raum%
}
\title{%
  Scalar-valued depth two Eichler-Shimura Integrals of Cusp Forms
}
\author{%
  Tobias Magnusson%
\and
  Martin Raum%
\thanks{The author was partially supported by Vetenskapsr\aa det Grant~2019-03551.}%
}

\begin{document}

\thispagestyle{scrplain}
\begingroup
\deffootnote[1em]{1.5em}{1em}{\thefootnotemark}
\maketitle
\endgroup


\begin{abstract}
\small
\noindent
{\tbf Abstract:}
Given cusp forms~$f$ and~$g$ of integral weight~$k \ge 2$, the
depth two holomorphic iterated Eichler-Shimura integral~$I_{f,g}$ is
defined by~${\int_\tau^{i\infty}f(z)(X-z)^{k-2}I_g(z;Y)\d{z}}$, where $I_g$ is
the Eichler integral of $g$ and $X,Y$ are formal variables.
We provide an explicit vector-valued modular
form whose top components are given by~$I_{f,g}$. We show that
this vector-valued modular form gives rise to a \emph{scalar-valued} iterated Eichler integral
of depth two, denoted by $\cE_{f,g}$, that can be seen as a higher-depth generalization of the scalar-valued
Eichler integral $\cE_f$ of depth one. As an aside, our argument provides an alternative
explanation of an orthogonality relation satisfied by period polynomials originally
due to Pa\c{s}ol-Popa. We show that $\cE_{f,g}$ can be expressed in terms of sums of products
of components of vector-valued Eisenstein series with classical modular forms
after multiplication with a suitable power of the discriminant modular form $\Delta$.
This allows for effective computation of $\cE_{f,g}$.
\\[.3\baselineskip]
\noindent
\textsf{\textbf{%
  iterated Eichler-Shimura integrals%
}}%
\noindent
\ {\tiny$\blacksquare$}\ %
\textsf{\textbf{%
  vector-valued modular forms%
}}%
\noindent
\ {\tiny$\blacksquare$}\ %
\textsf{\textbf{%
  Eichler cohomology%
}}
\\[.2\baselineskip]
\noindent
\textsf{\textbf{%
  MSC Primary:
  11F11
}}%
\ {\tiny$\blacksquare$}\ %
\textsf{\textbf{%
  MSC Secondary:
  11F30, 11F75
}}
\end{abstract}




\Needspace*{4em}
\addcontentsline{toc}{section}{Introduction}
\markright{Introduction}
\lettrine[lines=2,nindent=.2em]{\tbf I}{terated} Eichler-Shimura integrals have received a lot of interest in recent years.
For example they have been studied extensively by
Brown~\cite{brown_class_2018,brown_class_2020,brown_multiple_2017}, especially in
the context of iterated extensions of motives and multiple modular values. They
have also been related to what are known as higher order modular forms by
Diamantis~\cite{diamantis_modular_2022}. Furthermore, they are closely related to
the string theoretic notion of modular graph
functions~\cite{dorigoni_poincare_2022,dorigoni_poincare_2022-1,dhoker_integral_2019,dhoker_fourier_2018}.

In this paper, we examine iterated Eichler-Shimura integrals
of depth two in more detail. We first recall the definition of usual
Eichler integrals. There are two kinds of them --- scalar-valued and
polynomial-valued ones --- and they are given as follows. Let $f\in\rmS_k$, $k\in\ZZ_{\geq 2}$,
be a cusp form of level one, $X$ a formal variable, and~$\tau \in \HS$ in the Poincaré upper half plane, then the polynomial-valued
Eichler integral~$I_f(\,\cdot\,;X)$ and the scalar-valued Eichler integral~$\cE_f$ are given by
\begin{gather}
  \label{eq:depth_one_eich}
  I_f(\tau;X)
=
  \int_\tau^{i\infty} f(z)\,(X-z)^{k-2}\d{z}
\quad\tx{and}\quad
  \cE_f(\tau)
=
  \int_\tau^{i\infty} f(z)\,(\tau-z)^{k-2}\d{z}
\tx{.}
\end{gather}
Note that $I_f(\tau;\tau)=\cE_f(\tau)$. If $f,g\in\rmS_k$ are cusp forms
the depth two Eichler-Shimura integral is given by
\begin{gather}
  I_{f,g}(\tau;X,Y)=\int_\tau^{i\infty}f(z)\,(X-z)^{k-2}\,I_g(z;Y)\d{z}\tx{.}
\end{gather}
This Eichler-Shimura integral is well-understood and its definition can readily
be generalized to arbitrary depths~\cite{manin_iterated_2005,manin_iterated_2006}.
We provide a scalar-valued analogue $\cE_{f,g}$ to~$I_{f,g}$. It is given by
\begin{gather}
  \cE_{f,g}(\tau)=\int_\tau^{i\infty}f(z)\,\cE_g(z)\d{z}\tx{.}
\end{gather}
We are not aware of any previous occurrence of $\cE_{f,g}$ in the literature, but remark that its definition generalizes the one of~$\cE_f$ in a straightforward way. The first main
focus of this paper is to provide a connection between the geometrically motivated
$I_{f,g}$ and its classical counterpart $\cE_{f,g}$ paralleling the connection of
 $I_f$ to $\cE_f$. The second main focus of this paper is to provide a
framework that enables the effective computation of $\cE_{f,g}$ following the
approach taken in~\cite{ahlback_eichler_2022}.

We study $\cE_{f,g}$ and $I_{f,g}$ using the language of vector-valued modular forms. 
The work of Brown~\cite{brown_multiple_2017} 
implies via arguments of Mertens--Raum~\cite{mertens_modular_2021}
that $I_{f,g}$ is a component of a vector-valued modular form. However,
they do not specify this vector-valued modular form explicitly. This is the purpose
of our first theorem.

For an integer $\sfd\geq 0$, $\symd(X)$ denotes the $\sfd$th symmetric power
of the standard representation, whose representation space is the space $\CC[X]_\sfd$
of complex polynomials in $X$ of degree at most $\sfd$. We write $\symd(X,Y)$ for
the tensor product $\symd(X)\otimes\symd(Y)$, and $\CC[X,Y]_\sfd$ for its representation space.
There is a $\symd(X)$-invariant pairing on $\CC[X]_\sfd\times\CC[Y]_\sfd$ 
denoted by $\langle\,\cdot\,,\,\cdot\,\rangle$.
For the precise definitions, see Section~\ref{sec:prelim}.

Echoing $I_{f,g}$ and $\cE_{f,g}$, we consider the following
two related representations:
\begin{align}
  \begin{split}
  \widetilde{\rho}_{f,g}:\SL2\ZZ&\lra\GLbig{\CC[X,Y]_{k-2}\oplus\CC[X]_{k-2}\oplus\CC}
\tx{,}\\
  \ga
&\lmto
  \begin{pmatrix}
    \sym{k-2}(X,Y)(\ga) & \phi_{I_g}(\ga;Y)\cdot\sym{k-2}(X)(\ga) & \widetilde{\psi}_{f,g}(\ga)\\
    0 & \sym{k-2}(X)(\ga) & \phi_{I_f}(\ga)\\
    0 & 0 & 1
  \end{pmatrix}\tx{, and}\label{eq:fat_rep}
  \end{split}\\
  \begin{split}
  \rho_{f,g}:\SL2\ZZ&\lra\GLbig{\CC\oplus\CC[X]_{k-2}\oplus\CC}
\tx{,}\\
  \ga
&\lmto
  \begin{pmatrix}
    1 & -\phi_{I_g}^\vee(\ga) & \psi_{f,g}(\ga)\\
    0 & \sym{k-2}(X)(\ga) & \phi_{I_f}(\ga)\\
    0 & 0 & 1
  \end{pmatrix}\tx{,}\label{eq:thin_rep}
  \end{split}
\end{align}
where for a cusp form $h\in\rmS_k$ and a formal variable $W$, we set
\begin{gather*}
  \phi_{I_h}^\vee(\ga)(v)
=
  \big\langle \phi_{I_h}(\ga^{-1}),\, v \big\rangle
\tx{,\ }
  v\in\CC[X]_{k-2}
\tx{,}
\quad\tx{with}\quad
  \phi_{I_h}(\ga;W)
=
  \int^{i\infty}_{\ga(i\infty)}h(z)\,(W-z)^{k-2}\d{z}
\end{gather*}
and
\begin{gather*}
  \widetilde{\psi}_{f,g}(\ga)
=
  \int^{i\infty}_{\ga(i\infty)}f(z)\,(X-z)^{k-2}\,I_g(z;Y)\d{z}
\tx{,}\quad
  \psi_{f,g}(\ga)
=
  \int_{\ga(i\infty)}^{i\infty}f(z)\,\cE_g(z)\d{z}\tx{.}
\end{gather*}
Recall that $\phi_{I_h}$ is a parabolic $\symd(X)$-cocycle, see
Section~\ref{sec:prelim} and for example~\cite{kohnen_modular_1984}.

As mentioned previously, our first main theorem provides an explicit vector-valued modular form of type
$\widetilde{\rho}_{f,g}$ with $I_{f,g}$ as a component, thus amending Brown's results~\cite{brown_multiple_2017}. We write~$\rmM_k(\rho)$ for the space of weight~$k$ modular forms of type~$\rho$, which is defined in Section~\ref{ssec:vector_valued_modular_forms}.
\begin{maintheorem}
  \label{mainthm:fat_rep}
  Let $k\geq 2$ be an even integer, and let $f,g\in\rmS_k$. Then the arithmetic
  type $\widetilde{\rho}_{f,g}$ given by~\eqref{eq:fat_rep} is well-defined, and we have that
  \begin{gather*}
    \begin{pmatrix}
      I_{f,g}\\
      I_f\\
      1
    \end{pmatrix}
  \in
    \rmM_0\big( \widetilde{\rho}_{f,g} \big)\tx{.}
  \end{gather*}
\end{maintheorem}
The relation between $\widetilde{\rho}_{f,g}$ and $\rho_{f,g}$, and also between
$I_{f,g}$ and $\cE_{f,g}$, is provided by the contraction map $\pi:\symd(X,Y)\to\CC$,
given by
\begin{gather}
  \label{eq:contraction}
  \pi(p\otimes q)
=
  \langle p,q\rangle\tx{,}\qquad p\in\CC[X]_\sfd\tx{,}\quad q\in\CC[Y]_\sfd\tx{.}
\end{gather}
In particular, we provide the following theorem.
\begin{maintheorem}
  \label{mainthm:pp_1_sym_1}
  Let $k\geq 2$ be an even integer, and let $f,g\in\rmS_k$. Then the arithmetic
  type $\rho_{f,g}$ is well-defined, and the contraction
  map~\eqref{eq:contraction} induces a pushforward morphism
  $\pi_\ast:\widetilde{\rho}_{f,g}\to\rho_{f,g}$ given by
  $\pi_\ast(p,q,z)=(\pi(p),q,z)$. Furthermore, we have that
  \begin{gather*}
    \pi_\ast\circ\begin{pmatrix}I_{f,g}\\I_f\\1\end{pmatrix}
  =
    \begin{pmatrix}\cE_{f,g}\\I_f\\1\end{pmatrix}
  \in
    \rmM_0(\rho_{f,g})\tx{.}
  \end{gather*}
\end{maintheorem}

We next introduce a family of representations $\rho_{\phi_1,\phi_2}$. We
construct its members from pairs of parabolic $\symd(X)$-cocycles
$(\phi_1,\phi_2)$ satisfying a certain orthogonality relation, see
Section~\ref{ssec:the_ext} and specifically Theorem~\ref{theo:ortho_cond}.
Theorem~\ref{mainthm:pp_1_sym_1} implies that $\rho_{\phi_{I_f},-\phi_{I_g}}$ is
well-defined and equals~$\rho_{f,g}$.

\begin{mainremark}
Our characterization of pairs~$(\phi_1, \phi_2)$ for which~$\rho_{\phi_1,\phi_2}$ is well-defined in conjunction with our result that~$\rho_{f,g}$ is well-defined 
implies that~$\phi_{I_f}$ and~$\phi_{I_g}$ satisfies the aforementioned
orthogonality relation. This relation between cocycles has previously been obtained
by Pa\c{s}ol-Popa~\cite{pasol_modular_2013} in the language of period polynomials.
Theorem~\ref{mainthm:pp_1_sym_1} thus gives an alternate explanation of
Pa\c{s}ol-Popa's result in level~$1$. The case of general levels treated by
Pa\c{s}ol-Popa can be incorporated into our setting by introducing induced representations.
\end{mainremark}

In Section~\ref{sec:eis_sat}, we provide a framework
based on vector-valued Eisenstein series
that allows for effective
computation of modular forms of type $\rho_{\phi_1,\phi_2}$, that is, of forms
that transform like the representation~$\rho_{\phi_1,\phi_2}$
(see Section~\ref{ssec:the_ext} for the definition).
This builds upon the framework
developed in~\cite{ahlback_eichler_2022},
and in particular enables us to evaluate~$\cE_{f,g}$.

To prepare for the setup of this framework, we record that the representation $\rho_{\phi_1,\phi_2}$ features a function $\psi:\SL2\ZZ\to\CC$
that behaves similar to a $1$-cocycle. Specifically, it satisfies
\begin{gather*}
\begin{aligned}
  \psi(\ga_1\ga_2)
&=
  \psi(\ga_1)+\psi(\ga_2)+\big\langle\phi_2(\ga^{-1}),\phi_1(\ga)\big\rangle
\tx{,}\quad\ga_1,\ga_2\in\SL2\ZZ
\tx{,}
\\
  \psi(S) &= -\tfrac{1}{2} \big\langle\phi_2(S^{-1}),\phi_1(S)\big\rangle
\tx{,}\quad
  \psi(T) = 0
\tx{.}
\end{aligned}
\end{gather*}
We also remark that in the case of $\phi_1 = \phi_{I_f}$ and~$\phi_2 = -\phi_{I_g}$, this function
coincides with $\psi_{f,g}$.

Let $k$ be an integer, and let $(\phi_1,\phi_2)$ be a pair of parabolic cocycles
for which $\rho_{\phi_1,\phi_2}$ is well-defined. Then we define the vector-valued
Eisenstein series of weight $k\in\ZZ$, $k > 2 + \sfd$ and type $\rho_{\phi_1,\phi_2}$ by
\begin{gather}
  E_k(\tau;\rho_{\phi_1,\phi_2})
=
  \big( E_k^{[2]}(\tau;\phi_1,\phi_2),\;E_k^{[1]}(\tau;\phi_1),\;E_k \big)^T\tx{,}
\end{gather}
where $E_k$ is the classical Eisenstein series of weight $k$, and where
\begin{gather*}
  E_k^{[1]}(\tau;\phi_1)
=
  \sum_{[\ga]\in\Ga_\infty\bs\SL2\ZZ}\frac{\phi_1(\ga^{-1})}{(c\tau+d)^k}
\quad\tx{and}\quad
  E_k^{[2]}(\tau;\phi_1,\phi_2)
=
  \sum_{[\ga]\in\Ga_\infty\bs\SL2\ZZ}\frac{\psi(\ga^{-1})}{(c\tau+d)^k}\tx{.}
\end{gather*}
The series $E_k^{[1]}(\cdot;\phi_1)$ is called the generalized second order
Eisenstein series of type $(\symd(X),\bfone)$ associated to $(\phi_1,1)$ and was
a subject of study in~\cite{ahlback_eichler_2022}, in which it was shown that it
converges absolutely and locally uniformly on $\HS$ for $k>2+\sfd$, and where its
Fourier series expansion was provided. As for $E_k^{[2]}(\tau;\phi_1,\phi_2)$,
we show that it converges in the same region and provide its Fourier series expansion
in Theorem~\ref{theo:eis_conv_fe}.

To state our last main theorem, we write $\rmM_\bullet$ for the graded ring of
modular forms, and given a representation $\rho$, $\rmM_\bullet(\rho)$ for the
corresponding graded $\rmM_\bullet$-module of modular forms of type $\rho$.
Furthermore, if $M$ is a $\rmM_\bullet$-module,
$I\subseteq M$ is a submodule, and $f\in \rmM_\bullet$, then the saturation of
$I$ at $f$ is the $\rmM_\bullet$-module
\begin{gather*}
  (I:f^\infty)
=
  \big\{g\in M:\exists n\in\ZZ_{\geq 0}.\, f^ng\in I\big\}
\tx{.}
\end{gather*}
Recall also that given a parabolic $\symd(X)$-cocycle $\phi$, we have the representation
$\bfone\bp_{\phi^\vee}\symd(X):\SL2\ZZ\to\GL{\CC\oplus\CC[X]_\sfd}$ given by
\begin{gather*}
  (\bfone\bp_{\phi^\vee}\symd(X))(\ga)
=
  \begin{pmatrix}
    1 & \big\langle\phi(\ga^{-1}),\;\cdot\;\big\rangle\\
    0 & \symd(X)(\ga)
  \end{pmatrix}\tx{.}
\end{gather*}
We provide the following theorem.
\begin{maintheorem}
  \label{mainthm:sat}
  Let $\sfd\geq 0$ and $k_0>2+\sfd$ be integers and let
  $\phi_1,\phi_2$ be parabolic $\symd(X)$-cocycles making $\rho_{\phi_1,\phi_2}$
  well-defined (see Theorem~\ref{theo:ortho_cond}). Let also
  \begin{gather*}
    \rmE_{\geq k_0}(\rho_{\phi_1,\phi_2})
  =
    \mathrm{span}\;\rmM_\bullet\big\{E_k(\tau;\rho_{\phi_1,\phi_2})\,:\,k\geq k_0\big\}\tx{.}
  \end{gather*}
  Then
  \begin{gather*}
    \rmM_\bullet(\rho_{\phi_1,\phi_2})
  =
    \big(\rmE_{\geq k_0}(\rho_{\phi_1,\phi_2})
    +
    \iota\big(\rmM_\bullet(\bfone\boxplus_{\phi_2^\vee}\symd(X))\big)
    \colon
    \Delta^\infty\big)
  \end{gather*}
  where $\iota(f,g)=(f,g,0)^T$.
\end{maintheorem}
Since $\rmM_\bullet(\bfone\bp_{\phi_2^\vee}\symd(X))$ is described
in~\cite{ahlback_eichler_2022} as the saturation at $\Delta$ of generalized
second order Eisenstein series and classical modular forms, Theorem~\ref{mainthm:sat}
indeed implies that modular forms of type $\rho_{\phi_1,\phi_2}$ can be expressed
in terms of sums of products of (vector-valued) Eisenstein series with classical
modular forms, after multiplication with a suitable power of $\Delta$.

\section{Preliminaries}
\label{sec:prelim}
In this section, we define the notation we use throughout the paper, and revisit
the basic theory of Eichler cohomology, extensions of arithmetic types, and vector-valued
modular forms. For further details, we direct the reader to~\cite{westerholt-raum_products_2017} and~\cite{mertens_modular_2021}.

The special linear group of degree two over the integers, is given by
\begin{gather*}
  \SL2\ZZ=\Big\{\begin{pmatrix}a&b\\c&d\end{pmatrix}\in\ZZ^{2\times2}\,:\,ad-bc=1\Big\}\tx{.}
\end{gather*}
It is a fact that $\SL2\ZZ$ is generated by the matrices $S=(\begin{smallmatrix}0&-1\\1&0\end{smallmatrix})$
and $T=(\begin{smallmatrix}1&1\\0&1\end{smallmatrix})$. The parabolic subgroup
$\Ga_\infty\subseteq\SL2\ZZ$ is given by $\Ga_\infty=\langle T,-1\rangle$. We also let $U=TS$.

The upper-half plane $\HS$ is given by $\HS=\{\tau\in\CC:\Im(\tau)>0\}$. For
$\ga=\big(\begin{smallmatrix}a&b\\c&d\end{smallmatrix}\big)\in\SL2\ZZ$ and $\tau\in\HS$,
the Möbius action is given by
\begin{gather*}
  \ga\tau=\frac{a\tau+b}{c\tau+d}\tx{.}
\end{gather*}

Given a $\CC$-vector space $V$, a function $f:\HS\to V$, an integer $k$, and an element
$\ga=\big(\begin{smallmatrix}a&b\\c&d\end{smallmatrix}\big)\in\SL2\ZZ$, we define
a new function $f|_k\ga:\HS\to V$ by
\begin{gather*}
  \big(f\big|_k\ga\big)(\tau)=(c\tau+d)^{-k}f(\ga\tau)\tx{.}
\end{gather*}
We record that this gives rise to a right-action of $\SL2\ZZ$ on the space of
holomorphic functions from $\HS$ to $V$, which we call the slash-action.

Let $V$ be a $\CC$-vector space with a norm $\|\cdot\|$, and let $f:\HS\to V$ be
a function. If there exists a real number $a\in\RR$ such that for all $\ga\in\SL2\ZZ$
it holds uniformly in $\Re(\tau)$ that
\begin{gather}
  \label{eq:moderate_growth}
  \big\|\big(f\big|_k\ga\big)(\tau)\big\|=O(\Im(\tau)^a)\tx{ as }\Im(\tau)\to\infty\tx{,}
\end{gather}
we say that $f$ has moderate growth. If instead we have that for all $\ga\in\SL2\ZZ$ it
holds that
\begin{gather}
  \label{eq:cuspidal}
  \big\|\big(f\big|_k\ga\big)(\tau)\big\|\to 0\tx{ as }\Im(\tau)\to\infty\tx{,}
\end{gather}
we say that $f$ is cuspidal. Note that the conditions~\eqref{eq:moderate_growth}
and~\eqref{eq:cuspidal} are independent of the choice of norm on $V$.

Let $k$ be an integer. Then if $f:\HS\to\CC$ is a holomorphic function of moderate
growth satisfying that
\begin{gather*}
  f\big|_k\ga=f\tx{ for all }\ga\in\SL2\ZZ\tx{,}
\end{gather*}
we call $f$ a (scalar-valued) modular form of weight $k$. If $f$ is cuspidal, it is called a cusp
form. The set of modular forms of weight $k$ forms a $\CC$-vector space
denoted by $\rmM_k$. The corresponding subspace of cusp forms is denoted by $\rmS_k$. If $f\in\rmM_k$,
we let the conjugate modular form $f^c$ be given by
\begin{gather*}
  f^c(\tau)=\ol{f(-\ol{\tau})}\tx{.}
\end{gather*}
Note that if $f$ has the Fourier series expansion $f(\tau)=\sum_{n\geq 0}c_ne(n\tau)$,
then $f^c$ has the Fourier series expansion $f^c(\tau)=\sum_{n\geq 0}\ol{c_n}e(n\tau)$.

Finally, we denote the trivial representation of $\SL2\ZZ$ by $\bfone$, so that
$V(\bfone)=\CC$ and $\bfone(\ga)z=z$ for all $\ga\in\SL2\ZZ$ and $z\in\CC$.
\subsection{Cohomology and extensions of arithmetic types}

An arithmetic type is a finite-dimensional complex representation
of a subgroup $\Ga\subseteq\SL2\ZZ$. In the present paper, we restrict our scope
to arithmetic types of $\SL2\ZZ$.

If $\rho$ and $\sigma$ are arithmetic types, the $\CC$-vector space of $(\rho,\sigma)$-cocycles
is given by
\begin{gather*}
  \rmZ^1(\rho,\sigma)
=
  \big\{f:\SL2\ZZ\to\Hom(V(\rho),V(\sigma))\,:\,
  f(\ga_1\ga_2)=\sigma(\ga_1)f(\ga_2)+f(\ga_1)\rho(\ga_2)\big\}\tx{,}
\end{gather*}
and the subspace of $(\rho,\sigma)$-coboundaries is given by
\begin{gather*}
  \rmB^1(\rho,\sigma)
=
  \big\{f:\SL2\ZZ\to\Hom(V(\rho),V(\sigma))\,:\, 
  \exists h\in\Hom(V(\rho),V(\sigma)).\,f(\ga)=\sigma(\ga)h-h\rho(\ga)\big\}\tx{.}
\end{gather*}
If a cocycle $\phi\in\rmZ^1(\rho,\sigma)$ vanishes on every element of $\Ga_\infty$,
we call it parabolic. The space of all parabolic cocycles (or coboundaries)
is denoted by $\rmZ^1_\tx{pb}(\rho,\sigma)$ (or $\rmB^1_\tx{pb}(\rho,\sigma)$).
We will identify $\Hom(\bfone,V(\sigma))$ with $V(\sigma)$.

Given a parabolic cocycle $\phi\in\rmZ^1_\tx{pb}(\sigma,\rho)$, we define the arithmetic
type $\rho\bp_\phi\sigma$ by $V(\rho\bp_\phi\sigma)=V(\rho)\oplus V(\sigma)$ and
\begin{gather}
  (\rho\bp_\phi\sigma)(\ga)(v,v')=\big(\rho(\ga)v+\phi(\ga)v',\sigma(\ga)v'\big)\tx{.}
\end{gather}

The first parabolic cohomology group is the quotient
$\rmH^1_\tx{pb}(\rho,\sigma)=\rmZ^1_\tx{pb}(\rho,\sigma)/\rmB^1_\tx{pb}(\rho,\sigma)$. We record
that $\rmH^1_\tx{pb}(\sigma,\rho)$ is isomorphic to the group of parabolic extension
classes of $\sigma$ by $\rho$, denoted by $\Ext_\tx{pb}(\sigma,\rho)$. In particular,
the following map is a well-defined isomorphism of groups
\begin{gather*}
  \Ext_\tx{pb}(\sigma,\rho)
  \ni
  [0\to\rho\to\rho\bp_\phi\sigma\to\sigma\to0]
\lmto
  \phi+\rmB^1_\tx{pb}(\sigma,\rho)
  \in
  \rmH^1_\tx{pb}(\sigma,\rho)\tx{.}
\end{gather*}
In particular this means that given cocycles $\phi_1$ and $\phi_2$, it holds that
$\rho\bp_{\phi_1}\sigma$ is isomorphic to $\rho\bp_{\phi_2}\sigma$ if and only if
$\phi_1$ and $\phi_2$ are cohomologous.

Note that by the cocycle relations, an arbitrary cocycle $\phi\in\rmZ^1_\tx{pb}(\rho,\sigma)$
is fully determined by its value at $S$.

\subsection{Vector-valued modular forms}%
\label{ssec:vector_valued_modular_forms}

Let $k$ be an integer and let $\rho$ an be arithmetic type. Given a function $f:\HS\to V(\rho)$
and an element $\ga\in\SL2\ZZ$, we define $f\big|_k\ga:\HS\to V(\rho)$ by
\begin{gather*}
  \big(f\big|_{k,\rho}\ga\big)(\tau)=\rho(\ga^{-1})\big(f\big|_k\ga\big)(\tau)\tx{.}
\end{gather*}
A vector-valued modular form of type $\rho$ and weight $k$, is a holomorphic function
of moderate growth $f:\HS\to V(\rho)$ satisfying that
\begin{gather*}
  f\big|_{k,\rho}\ga=f\tx{ for all }\ga\in\SL2\ZZ\tx{.}
\end{gather*}
If a vector-valued modular form $f$ of type $\rho$ and weight $k$ is cuspidal,
we call it a cusp form. The space of vector-valued modular forms of weight $k$
and type $\rho$ is denoted by $\rmM_k(\rho)$. The corresponding subspace of
cusp forms is denoted by $\rmS_k(\rho)$.

We remark that scalar-valued modular forms of weight $k$ are the same
as vector-valued modular forms of type $\bfone$. That is, we have
the equalities $\rmM_k(\bfone)=\rmM_k$ and $\rmS_k(\bfone)=\rmS_k$.

For an arithmetic type $\rho$ we let the graded module of modular forms of type
$\rho$ be given by
\begin{gather*}
  {\rmM_\bullet(\rho)=\bigoplus_{k\in\ZZ}\rmM_k(\rho)}\tx{.}
\end{gather*}
If $M$ is an $\rmM_\bullet$-module, $I\subseteq M$ is a submodule of $M$, and
$f\in\rmM_\bullet$ then we recall that the saturation of $I$ at $f$ is given by
\begin{gather*}
  (I:f^\infty)=\big\{g\in M:\exists n\in\ZZ_{\geq 0}.\, f^ng\in I\big\}\tx{.}
\end{gather*}
\subsection{Symmetric powers}
Let $\sfd\geq 0$ be an integer. Then we let $\CC[X]_\sfd$ be the space of polynomials
with coefficients in $\CC$ of degree at most $\sfd$. We define the arithmetic type
$\symd(X)$ by $V(\symd(X))=\CC[X]_\sfd$, and
\begin{gather*}
  \symd(X)(\ga)p=p\big|_{-\sfd}\ga^{-1}=(-cX+a)^\sfd p\big(\mfrac{dX-b}{-cX+a}\big)
  \tx{,}\quad\tx{where }
  \ga=\big(\begin{smallmatrix}a&b\\c&d\end{smallmatrix}\big)\in\SL2\ZZ
  \tx{ and }
  p\in\CC[X]_\sfd\tx{.}
\end{gather*}
We remark that $\symd(X)$ is a model of the $\sfd$th symmetric power of the standard
representation of $\SL2\ZZ$, explaining the notation. The group ring
$\CC[\SL2\ZZ]$ acts linearly on $\CC[X]_\sfd$ by
\begin{gather*}
  \ga.p=\symd(X)(\ga)p\tx{ and }(c_1\ga_1+c_2\ga_2)p=c_1(\ga_1.p)+c_2(\ga_2.p)\tx{,}
\end{gather*}
where $\ga,\ga_1,\ga_2\in\SL2\ZZ$ and $c_1,c_2\in\CC$. As mentioned in the introduction
there exists a symmetric pairing $\langle\,\cdot\,,\,\cdot\,\rangle:\CC[X]_\sfd\times\CC[Y]_\sfd\to\CC$ given by
\begin{gather*}
  \langle p,q\rangle=\sum_{i=0}^\sfd(-1)^i\mbinom{\sfd}{i}^{-1}p_i q_{\sfd-i}\tx{,}
\end{gather*}
satisfying $\langle\ga.p,\ga.q\rangle=\langle p,q\rangle$ for any $p\in\CC[X]_\sfd$, $q\in\CC[Y]_\sfd$,
and $\ga\in\SL2\ZZ$. Since $\langle\,\cdot\,,\,\cdot\,\rangle$ is invariant and bilinear, we
have an equivariant contraction map $\pi:\CC[X,Y]_\sfd\to\CC$ given
by $\pi(p\otimes q)=\langle p,q\rangle$. There is also a related antisymmetric bilinear form
$\llangle\,\cdot\,,\,\cdot\,\rrangle:\CC[X]_\sfd^2\to\CC$ given by
\begin{gather*}
  \llangle p,q\rrangle=\big\langle T^{-1}.p-T.p,\,q\big\rangle\tx{.}
\end{gather*}
If $p = \sum p_i X^i \in\CC[X]_\sfd$, then we let $\ol{p} = \sum \ov{p_i} X^i \in\CC[X]_\sfd$.
Note that the action of~$\SL{2}{\ZZ}$ commutes with conjugation, so that $\ga.\ol{p}=\ol{\ga.p}$.

Henceforth, we will identify $\CC[X]_\sfd\otimes\CC[Y]_\sfd$ with $\CC[X,Y]_\sfd$;
the space of polynomials in $X$ and $Y$ of degree at most $\sfd$ in $X$ and $Y$. Furthermore,
we will use the shorthand notation $\symd(X,Y):=\symd(X)\otimes\symd(Y)$. This
coincides with the definition we used in the introduction.

%
\subsection{The Eichler-Shimura isomorphism}
We shall now briefly describe the Eichler-Shimura isomorphism between scalar-valued
cusp forms of weight $k\geq 2$ and $\rmZ^1_\tx{pb}(\bfone,\sym{k-2}(X))$. This exposition
follows~\cite{kohnen_modular_1984}. It holds for arbitrary integers~$\sfd \ge 0$, but the case relevant in this paper is~$\sfd = k - 2$.

The space of parabolic $(\bfone,\symd(X))$-cocycles can be completely
described as follows:
\begin{gather*}
  \rmZ^1_\tx{pb}\big(\bfone,\symd(X)\big)
=
  \big\{\phi:\SL2\ZZ\to\CC[X]_\sfd\,:\,
  \phi(-I)=\phi(T)=0\tx{ and }(1+S).\phi(S)=(1+U+U^2).\phi(S)=0\big\}\tx{,}
\end{gather*}
Note that $\rmZ^1_\tx{pb}\big(\bfone,\symd(X)\big)$ is closed under complex conjugation.
The space $\CC[X]_\sfd$ splits up into its ``even part'' and ``odd part''. That is, we have~$\CC[X]_\sfd=\CC[X]_\sfd^+\oplus\CC[X]_\sfd^-$ with
\begin{align*}
  \CC[X]_\sfd^+&=\big\{\textstyle\sum_{i=0}^\sfd p_iX^i\in\CC[X]_\sfd:p_{2j+1}=0\tx{ for }0\leq 2j+1\leq\sfd\big\}\\
  \CC[X]_\sfd^-&=\big\{\textstyle\sum_{i=0}^\sfd p_iX^i\in\CC[X]_\sfd:p_{2j}=0\tx{ for }0\leq 2j\leq\sfd\big\}
\tx{.}
\end{align*}
We further write
\begin{gather*}
  W_\sfd=\rmZ^1_\tx{pb}\big(\bfone,\symd(X)\big)(S)\tx{,}\quad
  W_\sfd^+=W_\sfd\cap\CC[X]_\sfd^+\tx{,}\quad\tx{and}\quad
  W_\sfd^-=W_\sfd\cap\CC[X]_\sfd^-\tx{.}
\end{gather*}
Let $f\in\rmS_k$ be a cusp form of weight $k\in\ZZ_{\geq 2}$. Then we define
the polynomial-valued Eichler integral of $f$ as
\begin{gather*}
  I_f(\tau;X)=\int_\tau^{i\infty}f(z)\,(X-\tau)^{k-2}\d{z}\tx{,}\quad\tx{where }\tau\in\HS\tx{.}
\end{gather*}
If the variable is understood from context, it is omitted from the notation. Note
that since $f$ vanishes at the cusp, $I_f(\tau)$ is well-defined. We also
see that for $\tau\in\HS$ we have that $I_f(\tau)\in\CC[X]_\sfd$. Given $\ga\in\SL2\ZZ$,
we let
\begin{gather*}
  \phi_{I_f}(\ga)=I_f\big|_{0,\sym{k-2}(X)}(1-\ga^{-1})=\int_{\ga(i\infty)}^{i\infty}f(z)\,(X-\tau)^{k-2}\d{z}\tx{.}
\end{gather*}
We have that $\phi_{I_f}\in\rmZ^1\big(\bfone,\sym{k-2}(X)\big)$ by construction, and
since elements of $\Ga_\infty$ stabilize the cusp we also have that
$\phi_{I_f}\in\rmZ^1_\tx{pb}\big(\bfone,\sym{k-2}(X)\big)$. We now set $r_f=\phi_{I_f}(S)\in W_{k-2}$ and
\begin{gather*}
  r_{f,n}
=
  \int_0^\infty f(i\!t)t^n\d{t}=\Ga(n+1)(2\pi)^{-n-1}\,\rmL(f,n+1)
\quad\tx{for }
  0\leq n\leq k-2\tx{,}
\end{gather*}
where $\rmL(f,\;\cdot\;)$ denotes the Hecke $\rmL$-function associated to $f$. By
the binomial theorem, we see that
\begin{gather*}
  r_f=\sum_{n=0}^{k-2}i^{-n+1}\mbinom{k-2}{n}r_n(f)X^{k-2-n}\tx{.}
\end{gather*}
Hence, we define $r^+_f\in W_{k-2}^+$ and $r^-_f\in W_{k-2}^-$ by
\begin{gather*}
  r^+_f=\sum_{\substack{0\leq n\leq k-2\\2\mid n}}(-1)^{n/2}\mbinom{k-2}{n}r_{f,n}X^{k-2-n}\quad\tx{and}\quad
  r^-_f=\sum_{\substack{0\leq n\leq k-2\\2\nmid n}}(-1)^{(n-1)/2}\mbinom{k-2}{n}r_{f,n}X^{k-2-n}\tx{,}
\end{gather*}
so that $r=r^-+ir^+$. We can now provide the Eichler-Shimura isomorphism.
\begin{untheorem}[Eichler-Shimura]
It holds that the maps
\begin{align*}
  \rmS_k\ni f&\mapsto r^-_f\in W^-_{k-2}\tx{ and}\\
  \rmS_k\oplus\CC\ni (f,z)&\mapsto r^+_f+z(X^{k-2}-1)\in W^+_{k-2}
\end{align*}
are isomorphisms of $\CC$-vector spaces.
\end{untheorem}
For proofs, see for example~\cite{shimura_introduction_1971} or~\cite{lang_introduction_1976}.
We now state an important result related to the bilinear form $\llangle\cdot,\cdot\rrangle$.
Recall first that the Petersson inner product $(\cdot,\cdot):\rmM_k\times\rmS_k\to\CC$ is given by
\begin{gather*}
  (f,g)=\int_{\SL2\ZZ\bs\HS}f(x+iy)\ol{g(x+iy)}y^k\frac{\d{x}\d{y}}{y^2}
\end{gather*}
We have the following theorem, due to Haberland and Pa\c{s}ol-Popa.
\begin{untheorem}[Haberland~\cite{haberland_perioden_1983} and Pa\c{s}ol-Popa~\cite{pasol_modular_2013}]
  Let $k\geq 2$ be an even integer, and let $f,g\in\rmS_k$. Then it holds that
  \begin{gather}
    \label{eq:haberland_pasol_popa}
    \llangle r_f,\ol{r_g}\rrangle=-6(2i)^{k-1}(f,g)\qquad\tx{and}\qquad\llangle r_f,r_g\rrangle=0\tx{.}
  \end{gather} 
\end{untheorem}

To contextualize the assumptions of Theorem~\ref{theo:ortho_cond} in the next section, we recall the following:
\begin{remark}
  Let $\sfd\geq 0$ be an even integer and let $e=X^\sfd-1\in W_\sfd$. Then
  the relations $(1+S).e=0$ and $(1+U+U^2).e=0$ imply that $\llangle e,q\rrangle=0$
  for all $q\in W_\sfd$. Conversely, the Eichler-Shimura isomorphisms imply that
  if an element $p\in W_\sfd$ satisfies that $\llangle p,q\rrangle=0$ for all
  $q\in W_\sfd$, then $p\in\CC\{e\}$. This means that
  $(W_\sfd/\CC\{e\},\llangle\cdot,\cdot\rrangle)$ is a non-degenerate symplectic
  vector space, and by applying the Eichler-Shimura isomorphisms to a basis of
  orthonormalized Hecke eigenforms one obtains an explicit isomorphism
  of $(W_\sfd/\CC\{e\},\llangle\cdot,\cdot\rrangle)$ with the standard complex
  symplectic vector space, given by
  \begin{gather*}
    \Big(\CC^{2D},(x,y)\mapsto x^T\begin{pmatrix}0&-1_D\\1_D&0\end{pmatrix}y\Big)\tx{,}
  \end{gather*}
  where $D=\dim(\rmS_{\sfd+2})$, and $1_D$ denotes the $D\times D$ identity matrix.
\end{remark}
\section{From extensions to Eichler-Shimura integrals}
\label{sec:ext_to_esint}
In this section, we show that polynomial- and scalar-valued depth two
Eichler-Shimura integrals can be regarded as components of vector-valued modular
forms of type $\widetilde{\rho}_{f,g}$ and $\rho_{f,g}$, respectively.

\subsection{The extension \texpdf{$\rho_{\phi_1,\phi_2}$}{rho\_phi1,phi2}}%
\label{ssec:the_ext}
Let $\sfd\geq 0$ be an even integer. For $\phi\in\rmZ^1_\tx{pb}\big(\bfone,\symd(X)\big)$,
we let
\begin{gather}
  \label{eq:duality}
  \phi^\vee(\ga)(v)=\big\langle\phi(\ga^{-1}),\,v\big\rangle\tx{.}
\end{gather}
The invariance of the pairing implies that $\phi^\vee\in\rmZ^1_\tx{pb}\big(\symd(X),\bfone\big)$,
and that~\eqref{eq:duality} defines an isomorphism of $\CC$-vector spaces from
$\rmZ^1_\tx{pb}\big(\bfone,\symd(X)\big)$ to $\rmZ^1_\tx{pb}\big(\symd(X),\bfone\big)$.

Given parabolic cocycles $\phi_1,\phi_2\in\rmZ^1_\tx{pb}\big(\bfone,\symd(X)\big)$ and
a function $\psi:\SL2\ZZ\to\CC$, we let
$\rho_{\phi_1,\phi_2,\psi}:\SL2\ZZ\to\GLbig{\CC\oplus\CC[X]_\sfd\oplus\CC}$ be given
by
\begin{gather}
  \rho_{\phi_1,\phi_2,\psi}(\ga)
=
  \begin{pmatrix}1&\phi_2^\vee(\ga) & \psi(\ga)\\
                 0&\symd(X)(\ga) & \phi_1(\ga)\\
                 0 & 0 & 1\end{pmatrix}\tx{,}\quad\ga\in\SL2\ZZ\tx{.}
\end{gather}
We have the following proposition.
\begin{proposition}
  \label{prop:TFAE}
  Let $\phi_1,\phi_2\in\rmZ^1_\mathrm{pb}\big(\bfone,\symd(X)\big)$, and let $\psi:\SL2\ZZ\to\CC$
  be a function. Then the following are equivalent
  \begin{enumerate}[(i)]
    \item $\rho_{\phi_1,\phi_2,\psi}$ is a representation,
    \item for all $\ga_1,\ga_2\in\SL2\ZZ$ it holds that $\psi(\ga_1\ga_2)=\psi(\ga_1)+\psi(\ga_2)+\phi_2^\vee(\ga_1)\phi_1(\ga_2)$,
    \item $(\phi_2^\vee,\psi)\in\rmZ^1_\mathrm{pb}\big(\symd(X)\bp_{\phi_1}\bfone,\bfone\big)$,
    \item $(\psi,\phi_1)^T\in\rmZ^1_\mathrm{pb}\big(\bfone,\bfone\bp_{\phi_2^\vee}\symd(X)\big)$.
  \end{enumerate}
  Furthermore, if any of these conditions hold, then
  \begin{gather*}
    \rho_{\phi_1,\phi_2,\psi}
  =
    \bfone\bp_{(\phi_2^\vee,\psi)}\big(\symd(X)\bp_{\phi_1}\bfone\big)
  =
    \big(\bfone\bp_{\phi_2^\vee}\symd(X)\big)\bp_{\big(\begin{smallmatrix}\psi\\\phi_1\end{smallmatrix}\big)}\bfone\tx{.}
  \end{gather*}
\end{proposition}
\begin{proof}
  Let $\rho=\rho_{\phi_1,\phi_2,\psi}$ and let $\ga_1,\ga_2\in\SL2\ZZ$. Then using
  the cocycle relations, and the fact that $\symd(X)$ is a representation, we find that
  \begin{gather*}
    \rho(\ga_1)\rho(\ga_2)
  =
    \begin{pmatrix}
      1 & \phi_2^\vee(\ga_1\ga_2) & \psi(\ga_2) + \phi_2^\vee(\ga_1)\phi_1(\ga_2) + \psi(\ga_1)\\
      0 & \symd(X)(\ga_1\ga_2) & \phi_1(\ga_1\ga_2)\\
      0 & 0 & 1
    \end{pmatrix}\tx{.}
  \end{gather*}
  We have that $\rho$ is a representation if and only if it is a homomorphism
  and therefore we see that (i) is equivalent to (ii).

  Continuing, we have that $(\phi_2^\vee,\psi)\in\rmZ^1_\tx{pb}\big(\symd(X)\bp_{\phi_1}\bfone,\bfone\big)$
  if and only if
  \begin{gather*}
    \big(\phi_2^\vee(\ga_2),\psi(\ga_2)\big)+\big(\phi_2^\vee(\ga_1),\psi(\ga_1)\big)\big(\symd(X)\bp_{\phi_1}\bfone\big)(\ga_2)
  =
    \big(\phi_2^\vee(\ga_1\ga_2),\psi(\ga_1\ga_2)\big)\tx{.}
  \end{gather*}
  On the other hand, we have that
  \begin{align*}
    &\big(\phi_2^\vee(\ga_2),\psi(\ga_2)\big)+\big(\phi_2^\vee(\ga_1),\psi(\ga_1)\big)\big(\symd(X)\bp_{\phi_1}\bfone\big)(\ga_2)\\
    &=\big(\phi_2^\vee(\ga_2)+\phi_2^\vee(\ga_1)\symd(X)(\ga_2),\,\psi(\ga_2)+\phi_2^\vee(\ga_1)\phi_1(\ga_2)+\psi(\ga_1)\big)\\
    &=\big(\phi_2^\vee(\ga_1\ga_2),\,\psi(\ga_1)+\psi(\ga_2)+\phi_2^\vee(\ga_1)\phi_2(\ga_2)\big)\tx{.}
  \end{align*}
  Hence, it is clear that (iii) is equivalent to (ii). In the same way, we find that
  (iv) is equivalent to (ii), which finishes the proof.
\end{proof}
For a fixed pair of cocycles $(\phi_1,\phi_2)$ there is at most one function $\psi$
satisfying the conditions of Proposition~\ref{prop:TFAE}.
\begin{theorem}
  \label{theo:ortho_cond}
  Let $\sfd\geq 0$ be an even integer and let $\phi_1,\phi_2\in\rmZ^1_\mathrm{pb}(\bfone,\symd(X))$.
  Then there exists a function $\psi$ satisfying the conditions of Proposition~\ref{prop:TFAE}
  if and only if $\llangle\phi_1(S),\phi_2(S)\rrangle=0$. Furthermore, if such a function
  exists, it is unique and is given by
  \begin{gather*}
    \psi(S)=-\mfrac{1}{2}\phi_2^\vee(S)\phi_1(S)\tx{.}
  \end{gather*}
\end{theorem}
\begin{proof}
  Let $\psi:\SL2\ZZ\to\CC$ be a function and let $\rho=\rho_{\phi_1,\phi_2,\psi}$.
  Recall that $\SL2\ZZ$ has the presentation $\langle S,T:S^4=(ST)^6=1\rangle$.
  Hence we have that $\rho$ is a representation if and only if
  \begin{gather*}
    \rho(S)^4=1\tx{ and }(\rho(S)\rho(T))^6=1\tx{.}
  \end{gather*}
  However, we find that
  \begin{align*} 
    \rho(S)^2
  &=
    \begin{pmatrix}
      1 & 0 & 2\psi(S) + \phi_2^\vee(S)\phi_1(S)\\
      0 & 1_\sfd & 0\\
      0 & 0 & 1
    \end{pmatrix}\tx{ and}\\
    (\rho(S)\rho(T))^3
  &=
    \begin{pmatrix}
      1 & 0 & 3\psi(S)+\phi_2^\vee((ST)^2)\phi_1(S)+\phi_2^\vee(ST)\phi_1(S)\\
      0 & 1_\sfd & 0\\
      0 & 0 & 1
    \end{pmatrix}\tx{,}
  \end{align*}
  where $1_\sfd$ is the $\sfd\times\sfd$ identity matrix. Hence $\rho$ is a
  representation if and only if
  \begin{align*}
      2\psi(S)+\phi_2^\vee(S)\phi_1(S)&=0\tx{ and}\\
      3\psi(S)+\phi_2^\vee((ST)^2)\phi_1(S)+\phi_2^\vee(ST)\phi_1(S)&=0\tx{.}
  \end{align*}
  We have that $\phi_2^\vee((ST)^2)=\phi_2^\vee(ST)+\phi_2^\vee(ST)ST$, that
  $\phi_2^\vee(ST)=\phi_2^\vee(S)T$, and that $T\phi_1(S)=\phi_1(TS)$. With~$U = TS$ this yields
  that the above is equivalent to
  \begin{align*}
    \psi(S)&=-\tfrac{1}{2}\phi_2^\vee(S)\phi_1(S)\tx{ and}\\
      \phi_2^\vee(S)(2\phi_1(U)+2\phi_1(U^2)-3\phi_1(S))&=0\tx{.}
  \end{align*}
  However, applying the identity $(1+U+U^2).\phi_1(S)=0$ to expand~$2\phi_1(S)$ and then~$(1+S).\phi_1(S) = 0$ to simplify the expression, we obtain that
  \begin{multline*}
    2\big(\phi_1(U)+\phi_1(U^2)\big)-3\phi_1(S)
  =
    2\big(\phi_1(U)+\phi_1(U^2)+U.\phi_1(S)+U^2.\phi_1(S)\big)-\phi_1(S)\\
  =
    2\big(T.\phi_1(S)+TST.\phi_1(S)+U^2.\phi_1(S)\big)-\phi_1(S)
  =
    2T.\phi_1(S)-\phi_1(S)\tx{,}
  \end{multline*}
  so that
  $\phi_2^\vee(S)\big(2\phi_1(U)+2\phi_1(U^2)-3\phi_1(S)\big)
  =
  \big\langle 2T^{-1}.\phi_2(S)-\phi_2(S),\,\phi_1(S)\big\rangle$.
To finish the proof we have to identify the right hand side with~$\llangle \phi_2(S),\phi_1(S)\rrangle = \big\langle (T^{-1}-T).\phi_2(S),\,\phi_1(S)\big\rangle$. To this end, note that for any element~$H \in \CC[\SL2\ZZ]$ the expression~$\langle H.\phi_2(S),\phi_1(S)\rangle$ only depends on the image of~$H$ in the double quotient~$(1+S) \backslash \CC[\SL2\ZZ \slash \pm I] \slash (1+S)$. Denoting equality in this quotient by~$\equiv$, the result follows from~$(1+U+U^2).\phi_1(S)=0$ and
\begin{gather*}
 T^{-1} - 1 + T
=
 S U^{-1} - 1 + U S^{-1}
=
  1 + (S U^{-1} - 1)(1 - US^{-1})
\equiv
  1 + (- U^{-1} - 1)(1 + U)
=
  - U^{-1} - 1 - U
\tx{.}
\end{gather*}
\end{proof}
If a pair of cocycles $(\phi_1,\phi_2)\in\rmZ^1_\tx{pb}\big(\bfone,\symd(X)\big)^2$ satisfies
that $\llangle\phi_1,\phi_2\rrangle=0$, we call it admissible. If
$(\phi_1,\phi_2)$ is admissible, then we omit $\psi$ from the notation and write
$\rho_{\phi_1,\phi_2}=\rho_{\phi_1,\phi_2,\psi}$.
\subsection{Depth two Eichler-Shimura integrals}
\label{ssec:d2esint}
In this section we describe depth two Eichler-Shimura integrals as components of
vector-valued modular forms of the types $\widetilde{\rho}_{f,g}$ and $\rho_{f,g}$,
where $f,g\in\rmS_k$, $k\in\ZZ_{\geq 2}$; defined in the introduction, see~\eqref{eq:fat_rep}
and~\eqref{eq:thin_rep}.

Recall that for $f,g\in\rmS_k$, $k\in\ZZ_{\geq 2}$, and indeterminates $X$ and $Y$
we have the depth two polynomial-valued Eichler-Shimura integral $I_{f,g}(\;\cdot\;;X,Y)$ and
the depth two scalar-valued Eichler-Shimura integral $\cE_{f,g}$, given by
\begin{gather*}
  I_{f,g}(\tau;X,Y)=\int_\tau^{i\infty}f(z)\,(X-z)^{k-2}\,I_g(z;Y)\d{z}
\quad\tx{and}\quad
  \cE_{f,g}(\tau)=\int_\tau^{i\infty}f(z)\,\cE_g(z)\d{z}\tx{,}\quad\tau\in\HS\tx{,}
\end{gather*}
where $I_g(\tau;Y)=\int_\tau^{i\infty}g(z)(Y-z)^{k-2}\d{z}$
and $\cE_g(\tau)=I_g(\tau;\tau)$ are the polynomial-valued
Eichler integral and the scalar-valued Eichler integral.

Recall that for $\ga\in\SL2\ZZ$ we have that
$\widetilde{\psi}_{f,g}(\ga;X,Y)=\int_{\ga(i\infty)}^{i\infty}f(z)\,(X-z)^{k-2}\,I_g(z;Y)\d{z}$. Let
also
\begin{align*}
  \phi_{I_g}\cdot\sym{k-2}(X):\SL2\ZZ&\to\Hom\big(\CC[X]_{k-2},\CC[X,Y]_{k-2}\big)
  \quad\tx{be given by }\\
  \big(\phi_{I_g}\cdot\sym{k-2}(X)\big)(\ga)p&{}=\phi_{I_g}(\ga)\cdot\sym{k-2}(X)(\ga)p
\tx{.}
\end{align*}
With this notation, we have that
\begin{gather}
  \label{eq:fat_rep_alt}
  \widetilde{\rho}_{f,g}
=
  \sym{k-2}(X,Y)
  \bp_{(\phi_{I_g}\cdot\sym{k-2}(X),\psi_{f,g})}
  \big(\sym{k-2}(X)\bp_{\phi_{I_f}}\bfone\big)\tx{.}
\end{gather}
We have that $(I_f,1)^T\in\rmM_0\big(\sym{k-2}(X)\bp_{\phi_{I_f}}\bfone\big)$, see~\cite{ahlback_eichler_2022},
and for the depth two polynomial-valued Eichler-Shimura integral we have the following theorem.
\begin{theorem}
  \label{theo:polynomial_eichler_shimura}
  Let $k\geq 2$ be an integer, and let $f,g\in\rmS_k$. Then the arithmetic
  type $\widetilde{\rho}_{f,g}$ given by~\eqref{eq:fat_rep} or~\eqref{eq:fat_rep_alt}
  is well-defined, and we have that
  \begin{gather*}
    \begin{pmatrix}
      I_{f,g}\\
      I_f\\
      1
    \end{pmatrix}
  \in
    \rmM_0(\widetilde{\rho}_{f,g})\tx{.}
  \end{gather*}
\end{theorem}
\begin{proof} 
  For convenience, we use the shorthand notation $|$ for $|_{0,\sym{k-2}(X,Y)}$,
  $\phi$ for $\phi_{I_g}\cdot\sym{k-2}(X)$, and $\rho$ for $\widetilde{\rho}_{f,g}$.
  Through direct calculation, we see that $\phi\in\rmZ^1_\tx{pb}\big(\sym{k-2}(X),\sym{k-2}(X,Y)\big)$.
  By a standard change of variables (see~\cite{kohnen_modular_1984}), we also obtain
  that
  \begin{gather}
  \label{eq:theo:polynomial_eichler_shimura}
    \phi_{f,g}(\ga,\tau)
  :=
r   \big(I_{f,g}\big|(1-\ga^{-1})\big)(\tau)
  =
    \widetilde{\psi}_{f,g}(\ga)+\phi(\ga)I_f(\ga^{-1}\tau)\tx{,}\quad\ga\in\SL2\ZZ\tx{.}
  \end{gather}
  To finish the proof of the transformation behavior, we record that
  \begin{gather*}
    \rho(\ga^{-1})\begin{pmatrix}I_{f,g}(\ga\tau)\\I_f(\ga\tau)\\1\end{pmatrix}
      -
      \begin{pmatrix}I_{f,g}(\tau)\\I_f(\tau)\\1\end{pmatrix}
  =
    \begin{pmatrix}
      I_{f,g}|(\ga-1)+\phi(\ga^{-1})I_f(\ga\tau)+\widetilde{\psi}_{f,g}(\ga^{-1})\\
      0\\
      0
    \end{pmatrix}
  =
    0\tx{.}
  \end{gather*}

  Let now $\ga_1,\ga_2\in\SL2\ZZ$. Then~\eqref{eq:theo:polynomial_eichler_shimura} implies that
  $\phi_{f,g}(\ga_1\ga_2,\tau)
  =
  \widetilde{\psi}_{f,g}(\ga_1\ga_2)+\phi(\ga_1\ga_2)I_f(\ga_2^{-1}\ga_1^{-1}\tau)$. On the
  other hand, we have that
  \begin{multline*}
    \phi_{f,g}(\ga_1\ga_2,\tau)
  =
    \big(I_{f,g}\big|(1-\ga_2^{-1})|\ga_1^{-1}\big)(\tau)+\big(I_{f,g}\big|(1-\ga_1^{-1})\big)(\tau)\\
  =
    \phi(\ga_1\ga_2)I_f(\ga_2^{-1}\ga_1^{-1}\tau)
    +
    \phi(\ga_1)\big(I_f(\ga_1^{-1}\tau)-\ga_2.I_f(\ga_2^{-1}\ga_1^{-1}\tau)\big)
    +
    \ga_1.\widetilde{\psi}_{f,g}(\ga_2)+\widetilde{\psi}_{f,g}(\ga_1)\tx{.}
  \end{multline*}
  However $(I_f\big|_{0,\sym{k-2}(X)}(1-\ga^{-1}))(\tau)$ is independent of $\tau$,
  and thus we have that
  \begin{gather*}
    \phi_{I_f}(\ga_2)
  =
    (I_f\big|_{0,\sym{k-2}(X)}(1-\ga_2^{-1}))(\ga_1^{-1}\tau)
  =
    I_f(\ga_1^{-1}\tau)-\ga_2.I_f(\ga_2^{-1}\ga_1^{-1}\tau)\tx{.}
  \end{gather*}
  We thus obtain the identity
  \begin{gather}
    \label{eq:psi_identity}
    \widetilde{\psi}_{f,g}(\ga_1\ga_2)
  =
    \ga_1.\widetilde{\psi}_{f,g}(\ga_2)+\widetilde{\psi}_{f,g}(\ga_1)+\phi(\ga_1)\phi_{I_f}(\ga_2)\tx{.}
  \end{gather}
  This leads immediately to the fact that
  $(\phi,\widetilde{\psi}_{f,g})\in\rmZ^1_\tx{pb}\big(\sym{k-2}(X)\bp_{\phi_{I_f}}\bfone,\,\sym{k-2}(X,Y)\big)$,
  and hence $\rho$ is well-defined.
\end{proof}

An analogous statement holds for~$\cE_{f,g}$.
\begin{theorem}%
\label{theo:scalar_eichler_shimura}
  Let $k\geq 2$ be an integer, and let $f,g\in\rmS_k$. Then the arithmetic
  type $\rho_{f,g}$ given by~\eqref{eq:thin_rep}
  is well-defined, and we have that
  \begin{gather*}
    \begin{pmatrix}
      \cE_{f,g}\\
      I_f\\
      1
    \end{pmatrix}
  \in
    \rmM_0(\rho_{f,g})\tx{.}
  \end{gather*}
\end{theorem}
\begin{proof}
The argument is parallel to the one for Theorem~\ref{theo:polynomial_eichler_shimura}.
\end{proof}
Combining Theorems~\ref{theo:polynomial_eichler_shimura} and~\ref{theo:scalar_eichler_shimura} with the contraction map
$\pi:\sym{k-2}(X,Y)\to\bfone$, we obtain our next theorem.
\begin{theorem}
  \label{theo:polynomial_to_scalar_eichler_shimura}
  Let $k\geq 2$ be an even integer, and let $f,g\in\rmS_k$. Then
  we have that
  \begin{gather*}
    \pi_\ast\circ\begin{pmatrix}I_{f,g}\\I_f\\1\end{pmatrix}
  =
    \begin{pmatrix}\cE_{f,g}\\I_f\\1\end{pmatrix}
  \tx{,}
  \end{gather*}
  where
  $\pi_\ast:\widetilde{\rho}_{f,g}\to\rho_{f,g}$ is the push-forward along the map~\eqref{eq:contraction}, given by
  $\pi_\ast(p,q,z)=(\pi(p),q,z)$.
\end{theorem}

\begin{proof}
  Let $\ga\in\SL2\ZZ$ and $p\in\CC[X]_{k-2}$, then
  \begin{gather*}
    \pi\big(\phi_{I_g}(\ga)\cdot\sym{k-2}(X)(\ga)p\big)
  =
    \big\langle\phi_{I_g}(\ga),\ga.p\big\rangle
  =
    \big\langle \ga^{-1}.\phi_{I_g}(\ga), p\big\rangle
  =
    -\big\langle\phi_{I_g}(\ga^{-1}),p\big\rangle\tx{,}
  \end{gather*}
  and therefore $\pi\circ(\phi_{I_g}\cdot\sym{k-2}(X))=-\phi_{I_g}^\vee$.
  That is, $\pi_\ast$ maps~$\wtd\rho_{f,g}$ to~$\rho_{f,g}$ as claimed.

  We consider the difference
  \begin{gather*}
    \pi_\ast\circ\begin{pmatrix}I_{f,g}\\I_f\\1\end{pmatrix}
  -
    \begin{pmatrix}\cE_{f,g}\\I_f\\1\end{pmatrix}
  =
    \begin{pmatrix}\pi\circ I_{f,g} - \cE_{f,g}\\0\\0\end{pmatrix}
  \in
    \rmM_0(\rho_{f,g})
  \tx{.}
  \end{gather*}
  Since the two bottom components vanish, we conclude that~$\pi\circ I_{f,g} - \cE_{f,g} \in \rmM_0$. Since~$f$ and~$g$ are cusp forms, the zeroth Fourier coefficient of~$\pi\circ I_{f,g} - \cE_{f,g}$ vanishes, and we obtain the equality stated in the theorem.

\end{proof}
\begin{remark}
Theorem~\ref{theo:scalar_eichler_shimura} implies that the pair of cocycles
$(\phi_{I_f},-\phi_{I_{g}})$ yields a representation, for which we see that $\rho_{f,g}=\rho_{\phi_{I_f},-\phi_{I_g}}$. Combining this with the orthogonality relation in Theorem~\ref{theo:ortho_cond} we obtain an alternate proof of Pa\c{s}ol-Popa's identity in level~$1$, as mentioned in the introduction.
\end{remark}
\section{Eisenstein series and saturation}
\label{sec:eis_sat}
Let $k$ and $\sfd\geq 0$ be integers, and let $(\phi_1,\phi_2)\in\rmZ^1_\tx{pb}\big(\bfone,\symd(X)\big)^2$
be admissible. In this section, we provide the Eisenstein series of type $\rho_{\phi_1,\phi_2}$
and weight $k$, converging absolutely and locally uniformly on $\HS$ for $k>2+\sfd$.
We also prove a more detailed version of Theorem~\ref{mainthm:sat}.
\subsection{Eisenstein series}
Let $\sfd\geq 0$ and $k>2+\sfd$ be even integers, and let
$(\phi_1,\phi_2)\in\rmZ^1_\tx{pb}\big(\bfone,\symd(X)\big)^2$ be an admissible pair of
parabolic cocycles. Then the weight $k$ Eisenstein series of type
$\rho_{\phi_1,\phi_2}$ is given by the series
\begin{gather*}
  E_k(\tau;\phi_1,\phi_2)
=
  \sum_{[\ga]\in\Ga_\infty\bs\SL2\ZZ}
    \begin{psmatrix}0\\0\\1\end{psmatrix}
      \Big|_{k,\rho_{\phi_1,\phi_2}}\ga
=
  \begin{pmatrix}
    \sum_{[\ga]\in\Ga_\infty\bs\SL2\ZZ}\psi(\ga^{-1})(c\tau+d)^{-k}\\
    E_k^{[1]}(\tau;\phi_1)\\
    E_k
  \end{pmatrix}\tx{,}
\end{gather*}
where $\psi$ is given as in Theorem~\ref{theo:ortho_cond},
$E_k(\tau)=\sum_{[\ga]\in\Ga_\infty\bs\SL2\ZZ}(c\tau+d)^{-k}$ and
$E_k^{[1]}(\tau;\phi_1)$ is the weight $k$ generalized second order Eisenstein
series of type $(\symd(X),\bfone)$ associated to $\phi_1$, given by
\begin{gather*}
  E_k^{[1]}(\tau;\phi_1)
=
  \sum_{[\ga]\in\Ga_\infty\bs\SL2\ZZ}\frac{\phi_1(\ga^{-1})}{(c\tau+d)^{k}}\tx{.}
\end{gather*}
By Lemma 3.10 of~\cite{ahlback_eichler_2022}, we have that $E_k^{[1]}(\tau;\phi_1)$ converges
absolutely and locally uniformly for $k>2+\sfd$. Its Fourier series expansion is
given in Theorem 3.8 of the same paper. We write
\begin{gather*}
  E_k^{[2]}(\tau;\phi_1,\phi_2)
=
  \sum_{[\ga]\in\Ga_\infty\bs\SL2\ZZ}\frac{\psi(\ga^{-1})}{(c\tau+d)^k}\tx{.}
\end{gather*}
Note that since $\phi_1$ and $\phi_2$ are parabolic, we have that
\begin{gather}
  \label{eq:psi_invariant}
  \psi(\pm T^m\ga T^n)
=
  \psi(T^m\ga T^n)
=
  \psi(\ga T^n)
=
  \psi(\ga)\tx{,}
\end{gather}
for any $m,n\in\ZZ$, so that $\psi$ descends to a function on the double
quotient $\Ga_\infty\bs\SL2\ZZ/\Ga_\infty$.

We now have the following theorem.
\begin{theorem}
  \label{theo:eis_conv_fe}
  Let $\sfd\geq0$ and $k>2+\sfd$ be even integers, let $(\phi_1,\phi_2)\in\rmZ^1_\mathrm{pb}\big(\bfone,\symd(X)\big)^2$
  be an admissible pair, and let $\psi$ be given as in Theorem~\ref{theo:ortho_cond}. Then
  $E_k^{[2]}(\tau;\phi_1,\phi_2)$ converges absolutely and locally uniformly
  on $\HS$ and has the following Fourier series expansion
  \begin{gather*}
    E_k^{[2]}(\tau;\phi_1,\phi_2)
  =
    \sum_{n\geq 1}e(n\tau)\frac{(-2\pi i)^k}{(k-1)!}
      \sum_{\substack{[\ga]\in\Ga_\infty\bs\SL2\ZZ/\Ga_\infty\\ [\ga]\neq 1}}
        \frac{n^{k-1}e(nd/c)\psi(\ga^{-1})}{c^k}\tx{.}
  \end{gather*}
\end{theorem}
\begin{proof}
  Reorganizing the defining series for $E_k^{[2]}(\tau;\phi_1,\phi_2)$,
  we obtain
  \begin{gather*}
    \sum_{[\ga]\in\Ga_\infty\bs\SL2\ZZ}\psi(\ga^{-1})(c\tau+d)^{-k}
  =
    \sum_{[\ga]\in\Ga_\infty\bs\SL2\ZZ/\Ga_\infty}\sum_{m\in\ZZ}\psi(T^{-m}\ga^{-1})(c(\tau+m)+d)^{-k}\tx{.}
  \end{gather*}
  However, by~\eqref{eq:psi_invariant} we have that $\psi(T^{-m}\ga^{-1})=\psi(\ga^{-1})$
  and thus we obtain the Fourier series expansion by applying Lipschitz' summation
  formula.

  As for convergence, we note that its enough to show that
  $|\psi(\ga^{-1})|\ll_\epsilon |c|^{\sfd+\epsilon}$ where $[\ga]\in\Ga_\infty\bs\SL2\ZZ/\Ga_\infty$
  and $[\ga]\neq 1$.

  To obtain this bound, we use a bijection between $\Ga_\infty\bs\SL2\ZZ/\Ga_\infty$
  and continued fractions.  We first have a bijection
  $s_1:\Ga_\infty\bs\SL2\ZZ/\Ga_\infty\to\QQ\cap[0,1)\cup\{\infty\}$ given by
  $s_1([\ga])=d'/c'+\lceil-d'/c'\rceil$ where $(c',d')=\mathrm{sgn}(c,d)\cdot(c,d)$.
  Let now $S$ be given by
  \begin{gather*}
    S
  =
    \big\{(0),()\big\}
    \cup
    \big\{(0,\alpha_1,\dots,\alpha_l) \,:\,
        l\geq 1,\alpha_l\geq 2,\forall 1\leq j< l.\, \alpha_j\geq 1\big\}\tx{,}
  \tx{.}
  \end{gather*}
  Then we have a bijection $s_2:S\to\QQ\cap[0,1)\cup\{\infty\}$, given by
  $s_2(())=\infty$ and $s_2(0,\alpha_1,\dots,\alpha_l)=[0;\alpha_1,\dots,\alpha_l]$,
  see~\cite{khinchin_continued_1964}. For convenience, we set~$\alpha_0 = 0$.

  Let $1\neq [\ga]\in\Ga_\infty\bs\SL2\ZZ/\Ga_\infty$ be arbitrary and
  let $\alpha=s_2^{-1}(s_1([\ga]))$. We then have that
  \begin{gather*}
    \psi(\ga^{-1})
  =
    \psi\big(T^{-\alpha_0}S^{-1}\cdots T^{(-1)^{l+1}\alpha_l}S^{-1}\big)
  =
    \psi\big(S^{2(l+1)}T^{-\alpha_0}S\cdots T^{(-1)^{l+1}\alpha_l}S\big)
  =
    \psi\big(T^{-\alpha_0}S\cdots T^{(-1)^{l+1}\alpha_l}S\big)\tx{.}
  \end{gather*}
  Let $\beta_i=(-1)^{l-i+1}\alpha_{l-i}$, $\delta_{-1}=I$, 
  and $\delta_i=T^{\beta_i}S\delta_{i-1}$, so that $[\delta_l]=[\ga^{-1}]$. We then
  have that
  \begin{gather*}
    \psi(\delta_i)-\psi(\delta_{i-1})
  =
    \psi(S)+\phi_2^\vee(S)\phi_1(\delta_{i-1})\tx{.}
  \end{gather*}
  We thus find that
  \begin{gather*}
    \psi(\delta_l)
  =
    \sum_{i=0}^l\big(\psi(\delta_i)-\psi(\delta_{i-1})\big)
  =
    \sum_{i=0}^l\phi_2^\vee(S)\big(\phi_1(\delta_{i-1})-\tfrac12\phi_1(S)\big)\tx{.}
  \end{gather*}
  Using the bound $\sum_{j=0}^\sfd\binom{d}{j}^{-1}\leq2+4/\sfd$, which holds
  for $\sfd\geq1$, we find that
  \begin{gather*}
    \big|\phi_2^\vee(S)\big(\phi_1(\delta_{i-1})-\tfrac12\phi_1(S)\big)\big|
  \leq
    4\|\phi_2(S)\|_1\big\|\phi_1(\delta_{i-1})-\tfrac12\phi_1(S)\big\|_1\tx{.}
  \end{gather*}
  Hence, we obtain the bound
  \begin{gather*}
    |\psi(\delta_l)|\leq 4\|\phi_2(S)\|_1\Big(
                                  \mfrac{l+1}{2}\|\phi_1(S)\|_1
                                  +
                                  \sum_{i=0}^l\|\phi_1(\delta_{i-1})\|_1\Big)\tx{.}
  \end{gather*}
  Lemma 3.10 in~\cite{ahlback_eichler_2022} tells us that for
  $\ga=\big(\begin{smallmatrix}a&b\\c&d\end{smallmatrix}\big)\in\SL2\ZZ$ with $|d|< |c|$
  we have that $\|\phi_1(\ga^{-1})\|_1\leq C_\sfd|c|^\sfd$, for
  a constant $C_\sfd\in\RR_{>0}$. In our case, we have that
  \begin{gather*}
    |d(\delta_i^{-1})/c(\delta_i^{-1})|
  =
    [\a_{l-i};\a_{l-i-1},\dots,\a_l]>1\qquad\tx{for }i<l\tx{,}
  \end{gather*}
  and thus we can apply the bound to $(\delta_i^{-1}S)^{-1}$. We have that
  $S^{-1}.\phi_1(\delta_i)=\phi_1(S^{-1}.\delta_i)-\phi_1(S)$
  and $\|S^{-1}.v\|_1=\|v\|_1$, and thus we obtain
  \begin{align*}
    \sum_{i=0}^l\|\phi_1(\delta_{i-1})\|_1
  \leq
    (l+1)\|\phi_1(S)\|_1+\sum_{i=0}^l\|\phi_1(S^{-1}\delta_{i-1})\|_1
  &\leq
    (l+1)\|\phi_1(S)\|_1+C_\sfd\sum_{i=0}^l|c(\delta_{i-1}^{-1}S)|^\sfd\\
  &=
    (l+1)\|\phi_1(S)\|_1+C_\sfd\sum_{i=0}^l|d(\delta_{i-1}^{-1})|^\sfd\tx{.}
  \end{align*}
  However, $|d(\delta_{i-1}^{-1})|$ is increasing in $i$, whence we obtain that
  \begin{gather*}
    \sum_{i=0}^l|d(\delta_{i-1}^{-1})|^\sfd
  \leq
    (l+1)|d(\delta_{l-1}^{-1})|
  =
    (l+1)|c(\ga)|^\sfd\tx{.}
  \end{gather*}
  In conclusion, we obtain that $|\psi(\ga^{-1})|
  \leq 4\|\phi_1(S)\|_1(l+1)(\tfrac32+C_\sfd|c|^\sfd)$.  Since $\alpha$ corresponds
  to a continued fraction, we have that $l+1\ll\log(|c|)$ and thus
  $|\psi(\ga^{-1})|\ll_\epsilon |c|^{\sfd+\epsilon}$ as desired.
\end{proof}
\subsection{Saturation}
Let $\sfd\geq 0$ and $k\geq 2$ be even integers, let $0\leq j\leq\sfd$ be an integer,
and let $(\phi_1,\phi_2)\in\rmZ^1_\mathrm{pb}(\bfone,\symd(X))^2$ be admissible.

In~\cite{ahlback_eichler_2022}, the authors of this paper and Ahlbäck introduced the $j$th
Eisenstein series of type $\bfone\bp_{\phi_2^\vee}\symd(X)$ and weight $k$, given
by
\begin{gather*}
  E_k\big(\tau;\bfone\bp_{\phi_2^\vee}\symd(X),j\big)
=
  \sum_{[\ga]\in\Ga_\infty\bs\SL2\ZZ}
    \begin{psmatrix}0\\(X-\tau)^j\end{psmatrix}
      \big|_{k,\bfone\bp_{\phi_2^\vee}\symd(X)}\ga
\tx{,}\quad\tau\in\HS\tx{.}
\end{gather*}
Let $k_0>2+\sfd$ be an integer. Then the $\rmM_\bullet$-module of Eisenstein
series of type $\bfone\bp_{\phi_2^\vee}\symd(X)$ is given by
\begin{gather*}
  \rmE_{\geq k_0}\big(\bfone\bp_{\phi_2^\vee}\symd(X)\big)
=
  \mathrm{span}\;\rmM_\bullet
    \big\{
      E_k\big(\tau;\bfone\bp_{\phi_2^\vee}\symd(X),j\big)\,:\,k\geq k_0\tx{, }0\leq j\leq\sfd
    \big\}
\tx{.}
\end{gather*}
Theorem 4.3 of~\cite{ahlback_eichler_2022} implies that for
$\sfd\geq0$ and $k_0>2+\sfd$ it holds that
\begin{gather}
  \label{eq:old_thm_vvmf}
  \big(\rmE_{\geq k_0}\big(\bfone\bp_{\phi_2^\vee}\symd(X)\big)+\iota(\rmM_\bullet)
  \colon
  \Delta^\infty\big)
=
  \rmM_\bullet\big(\bfone\bp_{\phi_2^\vee}\symd(X)\big)\tx{,}
\end{gather}
where $\iota(f)=(f,0)^T$. In Remark~4.4 of the same paper we mention that the
proof can be generalized to higher depths by induction. Following this approach,
we obtain our last theorem.
\begin{theorem}
  \label{theo:sat}
  Let $\sfd\geq 0$ be an even integer and $k_0>2+\sfd$ be an integer, and
  $(\phi_1,\phi_2)\in\rmZ^1_\mathrm{pb}\big(\bfone,\symd(X)\big)^2$ be admissible. Then
  it holds that
  \begin{gather*}
    \rmM_\bullet(\rho_{\phi_1,\phi_2})
  =
    \big(\rmE_{\geq k_0}(\rho_{\phi_1,\phi_2})
    +
    \iota_1\big(\rmE_{\geq k_0}(\bfone\bp_{\phi_2^\vee}\symd(X))\big)
    +
    \iota_2(\rmM_\bullet)
    \colon
    \Delta^\infty\big)\tx{,}
  \end{gather*}
  where $\iota_1(f,g)=(f,g,0)^T$ and $\iota_2(f)=(f,0,0)^T$.
\end{theorem}
\begin{proof}
Consider the following diagram
\[
  \begin{tikzcd}
    \rmM_\bullet(\bfone\bp_{\phi_2^\vee}\symd(X))
    \arrow[hookrightarrow,"\mu_1"]{r}\arrow["="]{d}&
    (\rmE_{\geq k_0}(\rho_{\phi_1,\phi_2})+\mu_1(\rmM_\bullet(\bfone\bp_{\phi_2^\vee}\symd(X)))\colon\Delta^\infty)
    \arrow["\nu_1"]{r}\arrow[hookrightarrow,"f\mapsto f"]{d}&
    (\rmE_{\geq k_0}\colon\Delta^\infty)
    \arrow["="]{d}
    \\
    \rmM_\bullet(\bfone\bp_{\phi_2^\vee}\symd(X))
    \arrow[hookrightarrow,"\mu_2"]{r}&
    \rmM_\bullet(\rho_{\phi_1,\phi_2})
    \arrow["\nu_2"]{r}&
    \rmM_\bullet
  \end{tikzcd}
\]
where $\mu_i(f,g)=(f,g,0)^T$ and $\nu_i(f,g,h)=h$ for $i\in\{1,2\}$. It is clear
that the rows are exact and that the diagram commutes. We have that
$\mathrm{im}(\nu_1)\subseteq\mathrm{im}(\nu_2)$ and thus the map
$\theta:\mathrm{coker}(\nu_1)\to\mathrm{coker}(\nu_2)$ given by
$\theta(h+\mathrm{im}(\nu_1))=h+\mathrm{im}(\nu_2)$ is well-defined. Since
$E_k=\nu_1(E_k(\cdot;\phi_1,\phi_2))$ it is also injective. The
Four Lemma now implies that the vertical map in the middle column is surjective
and thus an equality. Finally, we conclude the proof by
applying~\eqref{eq:old_thm_vvmf} and observing that
\begin{multline*}
  \big(\rmE_{\geq k_0}(\rho_{\phi_1,\phi_2})
  +
  \mu_1\big(
    (\rmE_{\geq k_0}(\bfone\bp_{\phi_2^\vee}\symd(X))
    +
    \iota(\rmM_\bullet)
    \colon
    \Delta^\infty)
       \big)
  \colon
  \Delta^\infty\big)\\
=
  \big(\rmE_{\geq k_0}(\rho_{\phi_1,\phi_2})
  +
  \iota_1\big(\rmE_{\geq k_0}(\bfone\bp_{\phi_2^\vee}\symd(X))\big)
  +
  \iota_2(\rmM_\bullet)
  \colon
  \Delta^\infty
  \big)
\tx{.}
\end{multline*}
\end{proof}

\ifbool{nobiblatex}{%
  \bibliographystyle{alpha}%
  \bibliography{bibliography.bib}%
}{%
  \Needspace*{4em}
  \printbibliography[heading=none]
}


\Needspace*{3\baselineskip}
\noindent
\rule{\textwidth}{0.15em}

{\noindent\small
Chalmers tekniska högskola och Göteborgs Universitet,
Institutionen för Matematiska vetenskaper,
SE-412 96 Göteborg, Sweden\\
E-mail: \url{tobmag@chalmers.se}\\
Homepage: \url{https://tobiasmagnusson.com}
}\vspace{.5\baselineskip}

{\noindent\small
Chalmers tekniska högskola och G\"oteborgs Universitet,
Institutionen f\"or Matematiska vetenskaper,
SE-412 96 G\"oteborg, Sweden\\
E-mail: \url{martin@raum-brothers.eu}\\
Homepage: \url{http://martin.raum-brothers.eu}
}


\ifdraft{%
\listoftodos%
}

\end{document}
